\documentclass[12pt,leqno]{amsart}
\usepackage{amsmath, amssymb, amscd, amsfonts,   stmaryrd, turnstile, mathrsfs, color}

\definecolor{dullmagenta}{rgb}{0.4,0,0.4}

\definecolor{darkblue}{rgb}{0,0,0.4}

\usepackage[colorlinks=true, pdfstartview=FitV, linkcolor=red, citecolor=blue, urlcolor=darkblue]
{hyperref}

\newtheorem{theorem}{Theorem}[section]
\newtheorem{lemma}[theorem]{Lemma}
\newtheorem{proposition}[theorem]{Proposition}
\newtheorem{corollary}[theorem]{Corollary}

\theoremstyle{definition}
\newtheorem{definition}[theorem]{Definition}
\newtheorem{example}[theorem]{Example}
\newtheorem{remark}[theorem]{Remark}

\begin{document}

\title[Computing the closure of a support]{Computing the closure of a support}

\author[G. Picavet and M. Picavet]{Gabriel Picavet and Martine Picavet-L'Hermitte}

\address{Math\'ematiques, 8 Rue du Forez, 63670 Le Cendre, France}

\email{picavet.mathu(at) orange.fr}
\email{}
\begin{abstract}

When $E$ is an $R$-module over a commutative unital ring $R$, the Zariski closure of its support is of the form $\mathrm V(\mathcal O(E))$ where $\mathcal O(E)$ is a unique radical ideal. We give an explicit form of $\mathcal O(E)$ and study its behavior under various operations of algebra. Applications are given, in particular for ring extensions of commutative unital rings whose supports are closed. We provide some applications to crucial and critical ideals of ring extensions.
\end{abstract} 



\maketitle

\section{Introduction}
We consider modules $E$ over  commutative unital rings
$R$ and  (ring) extensions of commutative unital rings. 
 
Let $\mathcal P$ be a property of modules or ring extensions. In order to measure the obstruction to $\mathcal P$, with the help of an ideal, some commutative algebraists have found useful to consider, in some contexts, localizations $E_a =E\otimes_RR_a$ of an $R$-module $E$, with respect to the multiplicatively closed subsets $\{a^n \mid n\in \mathbb N\}$, where $a\in R$. They then consider the set of all elements $a\in R$, such that $E_a$ verifies $\mathcal P$, instead of the subset of $\mathrm{Spec}(R)$, whose elements $P$  are such that $E_P$ does not verify $\mathcal P$. In this paper, the property involved is the zero property, as did S. Oda in a paper devoted to birational extensions of commutative integral domains \cite{O}. This is developed in the next sections, where we introduce the Oda ideal of modules and of ring extensions, which allows us to compute the closure of a support and other various open subsets of a spectrum, associated to ring morphisms, especially injective flat epimorphisms. We also examine the behavior of the Oda ideal with respect to different algebraic operations.
We then give an interpretation of crucial or critical ideals associated to a ring extension.
 The case of  Nagata extensions is examined.

 We first give some notation and recalls. 

(1) The annihilator of an $R$-module $E$ is denoted by $0:_RE$ or $0:E$. If $R\subseteq S$ is a (ring) extension, $0:_R (S/R)$ is the conductor $(R:S)$ of the extension, the greatest ideal shared by $R$ and $S$. 
 
(2) We say that an $R$-module $E$ is almost finitely generated (afg) over $R$ if there is a ring morphism $R \to S$ such that $E$ is an $S$-module of finite type, inducing on $E$ the original structure of $R$-module. A ring extension $R \subseteq S$ is said (module) finite if the $R$-module $S$ is of finite type and of finite type if $S$ is an $R$-algebra of finite type. 

(3) In this paper, a compact topological space does not need to be separated (Hausdorff).
 If $R$ is a ring and $I$ an ideal of $R$, we denote by $\mathrm V(I)$ the subset $\{P\in \mathrm{Spec}(R) \mid I \subseteq P\}$ of $\mathrm{Spec}(R)$ and by $\mathrm D(I)$ its complement. We recall that the spectrum $\mathrm{Spec}(R)$ of a ring $R$ is a topological space whose closed subsets are the sets $\mathrm V(I)$ where $I$ is an ideal of $R$. As usual this topology is called the Zariski topology. Note that a closed set $\mathrm V(I)$ is determined by $\sqrt I$, because $\mathrm V(I) =\mathrm V(\sqrt I)$ and  $\mathrm V (I) = \mathrm V (J)$ if and only if $\sqrt I = \sqrt J$. This remark justifies the title of the paper. We will use the following result: if $X$ is a subset of $\mathrm{Spec}(R)$, its Zariski closure $\overline X$ is $\mathrm V(\cap [P \mid P\in X])$. A useful topology on $\mathrm{Spec}(R)$ is the opposite topology, introduced by M. Hochster, called the o-topology, also termed the flat topology. Its closed sets are the spectral images of flat ring morphisms $R \to S$ and a basis of its open subsets is the set of all $\mathrm V(I)$ where $I$ is a finitely generated ideal of $R$ \cite[Chapitre IV]{FLAT}. We will only need the following facts: the set $\mathrm V(I)$ is open in the flat topology if $I$ is an ideal of finite type and the flat topology is compact. 

There is another topology on $\mathrm{Spec}(R)$; that is, the constructible topology. Its closed sets are the pro-constructible subsets of $\mathrm{Spec}(R)$, that are the  spectral images of ring morphisms $R\to T$. This topology is compact and separated (see \cite[7.2.11, p.331]{EGA}) and finer than the Zariski topology and the flat topology, because if $I$ is an ideal of $R$, $\mathrm V(I) $ is the spectral image of $\mathrm{Spec}(R/I)$ and a closed set in the flat topology is a spectral image \cite[Chapitre IV]{FLAT}.
If $f: R \to T$ is a ring morphism, we denote by ${}^af$ the natural map $\mathrm{Spec}(T) \to \mathrm{Spec}(R)$. This map is continuous for all the above topologies and we set $\mathrm{Spec}(T,R):= {}^af(\mathrm{Spec}(T))$.
Note that if $Y\subseteq \mathrm{Spec}(T)$ is pro-constructible, so is ${}^af(Y)$.

(4) We denote by $\mathrm{Supp}_R(E):= \{P\in\mathrm {Spec}(R)\mid E_P\neq 0\}$, (also denoted by $\mathrm{Supp}(E)$) the support of an $R$-module $E$. 
 It is known that $\mathrm{Supp}_R(E)$ is stable under specialization; that is, if $P\subseteq Q$ are in $\mathrm{Spec}(R)$ and $P\in \mathrm{Supp}_R(E)$, then $Q\in \mathrm{Supp}_R(E)$.
  Let $X\subseteq \mathrm{Spec}(R)$. The specialization of $X$ is denoted by $X^{\uparrow}$ and we have $X^{\uparrow}:=\{Q\in\mathrm{Spec}(R)\mid P\subseteq Q$ for some $P\in X\}$. 
 
 If $E$ is an $R$-module of finite type, it is also well known that $\mathrm{Supp}_R(E)$
 
 \noindent $ = \mathrm V(0:_RE)$. Moreover, if $E$ is of finite presentation, 
 $\mathrm{Supp}_R(E)=\mathrm V (I)$ where $I$ is a finitely generated ideal of $R$. Let $\{x_1,\ldots,x_n\}$ be a finite set of generators of $I$. We get that $\mathrm{Spec}(R)\setminus\mathrm V(I)=\cup_{i=1}^n\mathrm D (x_i)$, so that $\mathrm{Spec}(R)\setminus \mathrm V(I)$ is compact and $\mathrm{Supp}_R(E)$ is open in the flat topology. 
 
(5) We will need the notion of base change. Let $R\to S$ and $R\to T$ be ring morphisms. The ring morphism $T \to T\otimes_R S$ is called the ring morphism deduced from $R\to S$ by the base change $R\to T$. Now a property of ring morphism is called universal if it is stable under any base change.

(6) The category of commutative unital rings has epimorphisms that are not necessarily surjective, for example localizations $R \to R_\Sigma$ of a ring $R$ with respect to a multiplicatively closed subset $\Sigma$. Actually, these epimorphisms are flat epimorphisms.

(7) We recall the Stokes formula \cite[4.3]{S}: if $E$ and $M$ are $R$-modules and $M$ is of finite type, then $M\otimes_R E = 0$ if and only if  $E = (0:M)E$.

(8) $\mathrm{Nil}(R)$ denotes the set of all nilpotents elements of a ring $R$.
\section{Oda ideals}
 
Let $E$ be an $R$-module. If $P$ is an element of $\mathrm{Spec}(R)$, then $P \in \mathrm{Supp}(E)$ if and only if there is some $x \in E$; such that $0:x \subseteq P$. It follows that $\mathrm{Supp}(E) =\cup [\mathrm V(0:x) \mid x \in E]$, whence $\mathrm{Supp}(E)$ is a union of Zariski closed sets of $\mathrm{Spec}(R)$ and is not in general a closed set. For example, if $X$ is a non-closed subset of $\mathrm{Spec}(\mathbb Z)$ and $M = \oplus [\mathbb Z/P \mid P\in X]$, then the support of $M$ is $X$ 
 \cite[Corollaire, p. 133]{B1}.

We begin with the calculation of the Zariski closure of the support of a module. We recall that the (weak)assassinator $\mathrm{Ass}_R(E)$ of an $R$-module $E$ is the set of prime ideals (associated prime ideals) $P$ of $R$ that are minimal in the set of all prime ideals containing the annihilator of some $x\in E$ \cite[D\'efinition 1.1, p. 92]{L}.
  
   We will drop the suffix $R$ if no confusion can happen. 
\begin{lemma}\label{ass} Let $E$ be an $R$-module, then $\overline{\mathrm{Supp}_R(E)} =\overline{\mathrm{Ass}_R(E)}$ and $\mathrm{Supp}_R(E) =\cup [\mathrm V(P) \mid P\in \mathrm{Ass}_R(E)]\supseteq \mathrm{Ass}_R(E)$.
  \end{lemma}
  
\begin{proof} Observe that $\mathrm{Supp}(E)$ is the specialization of $\mathrm{Ass}(E)$ \cite [Propri\'et\'e 1.8, p.93]{L}, {\it i.e.} a prime ideal $P$ of $R$ belongs to $\mathrm{Supp}(E)$ if and only if it contains some element of $\mathrm{Ass}(E)$. Now we know that for a subset $X$ of $\mathrm{Spec}(R)$, we have $\overline X =\mathrm{V}( \cap [P \mid P\in X])$. The conclusion follows.
  \end{proof}
  
 \begin{proposition}\label{clo} Let $E$ be an $R$-module, then :
 
\centerline{$\overline{\mathrm{Supp}_R(E)}=\mathrm V(\cap[\sqrt{0:_R x}\mid x\in E]) \subseteq \mathrm V(0:_RE)$.}
 \end{proposition} 
 
 \begin{proof}  
 Lemma~\ref{ass} implies that  $\overline{\mathrm{Supp}_R(E)} =\mathrm V (\cap[ P\mid P\in \mathrm{Ass}(E)])$. 
To show the  equation, we introduce the ideal $I := \cap [P\mid P\in \mathrm{Ass}(E)]$. To conclude, we need only to prove that $I= \cap [\sqrt{0:_R x} \mid x\in E]$. This is a consequence of \cite[Propri\'et\'e 1.4, p.92]{L}, which asserts that $a\in R$ belongs to $I$ if and only if the map $E\to E$ defined by $x \mapsto ax$,  is pointwise nilpotent. 
 \end{proof}
  
We now introduce an ideal linked to the annihilator of an $R$-module $M$, first defined by Oda for some ring extensions \cite{O}. We generalize Oda's definition to modules over  arbitrary rings, with a slight change.

\begin{definition} The Oda ideal $\mathcal O_R (E)$ (or $\mathcal O (E)$) of an $R$-module $E$ is the set of all $a\in R$ such that $E_a=~0$. 
Consequently, $E_a=0$ if and only if $\mathrm{Ass}_R(E)\subseteq \mathrm V(a)$ \cite[Propri\'et\'es 1.2 and 1.6, p.92]{L}.

It  is easy to check that 

\centerline{$\mathcal O(E) = \cap [\sqrt{0:_Rx} \mid x \in E] = \{a \in R \mid \forall x \in E \, \exists n \in \mathbb N  \,a^nx =0\}.$}
 It follows that  $\mathcal O (E)$ is a radical ideal containing  $0:_RE$ and $\mathrm{Nil}(R)$.  Moreover, if $a \in \mathcal O(E)$, we have $\cap [Ra^n \mid n \in \mathbb N] \subseteq 0:_R E$. 
  In fact, setting $\mathrm{Nil}(E):=\{a\in R\mid \forall x \in E \, \exists n \in \mathbb N \,a^nx =0\}$, \cite[Propri\'et\'e 1.4, p.92]{L} asserts that $\mathrm{Nil}(E)=\cap[P \mid P\in \mathrm{Ass}_R(E)]$, so that $\mathrm{Nil}(E)= \mathcal O(E)$.

If $R\subseteq S$ is an extension, we set $\mathcal O(R,S) := \mathcal O (S/R)$; so that $\mathcal O(R,S) = \cap [\sqrt{R:s} \mid s\in S]$.

Therefore, $(R:S) \subseteq \mathcal O(R,S) = \{a\in R \mid R_a= S_a\}$.  
 We note that $\mathcal O(R,S) = R$ if and only if $R=S$. 

If $f: R\to S$ is a ring morphism, we may also consider:
$\mathcal O(f) := \{a\in R \mid R_a \to S_a $ is an isomorphism$\}$.
\end{definition}
 The following result is obvious.
\begin{lemma}If $R\subseteq S$ is a ring extension, then $\mathcal O(R,S)/(R:S) = \mathcal O(R/(R:S),S/(R:S))$.
\end{lemma}

\begin{example}\label{EXA} Suppose that $R$ is an absolutely flat ring (Von Neumann regular). Then each ideal of $R$ is semiprime, because each element $x\in R$ has a quasi-inverse $x'$, such that $x^2x'= x$ and in particular $R$ is a reduced ring. Now if $E$ is an $R$-module, whence flat over $R$, we have $\mathcal O(E) =\cap [0:x \mid x \in E] =0:E$.
\end{example}

\begin{remark}\label{EXA1}  Let $R\subseteq S$ be a ring extension. For $s\in S\setminus R$, we set $\mathrm d(s) := \{r\in R \mid rs\in R\} = R:_Rs$, the so-called ``denominator ideal'' of $s$. We then have $\mathcal O(R,S) = \cap [\sqrt {\mathrm d(s)} \mid s\in S \setminus R]$. Therefore, $a\in \mathcal O (R,S)$ if and only if for each $s\in S$, there is  some integer $n \geq 0$, such that $a^ns = r\in R$. 
\end{remark} 

If $E$ is an $R$-module, its (Nagata) idealization $R (+)E$ is defined as the set $R\times E$, endowed with the usual addition and the multiplication defined by $(r,x)(s,y) = (rs, ry + sx)$. Then $R (+)E$ is a commutative ring, and $R\to R (+)E$, defined by $r\mapsto (r,0)$ is a ring extension.

We observe that if $E$ is an $R$-module with idealization $R(+) E$, then $\mathcal O(E) = \mathcal O (R, R(+)E)$, because $E\simeq (R(+)E)/R$. 

\begin{theorem}\label{cars} Let $E$ be an $R$-module, then $\overline{\mathrm{Supp}_R(E)} = \mathrm V(\mathcal O_R(E)) \subseteq \mathrm V(0:_R E)$. 
\end{theorem}  

\begin{proof} The statement holds because of Proposition \ref{clo} and $ \mathcal O(E) = \cap [\sqrt{0:_R x} \mid x\in E]$.
\end{proof}

\begin{corollary}If $\ \mathrm{Supp}_R(E)$ is Zariski closed, then $\mathrm{Supp}_R(E) = \mathrm V(\mathcal O_R(E))$; so that, $E_P= 0$ for $P\in \mathrm{Spec}(R)$ if and only if there  is some $a\in R\setminus P$ such that $E_a=0$. 
\end{corollary}

\begin{remark}\label{cars1} Let $E$ be an $R$-module. 

(1) The above Theorem implies that $\mathcal O_R(E)=\mathrm{Nil}(R)$ if and only if $\mathrm{Supp}(E)$ is Zariski dense in $\mathrm{Spec}(R)$.

(2) We now examine when a support is the whole spectrum. Obviously, $\mathrm{Nil}(R)\subseteq\sqrt{0:_RE}$.    
  
Assume that $\mathrm{Supp}_R(E)=\mathrm{Spec}(R)$. If $\mathrm{Nil}(R)\subset\sqrt{0:_RE}$, there exists some $a\in\sqrt{0:_RE}\setminus\mathrm{Nil}(R)$, and then some $P\in\mathrm{Spec}(R)$ such that $a\not\in P$ and some integer $n$ such that $a^n\in 0:_RE$. It follows that $a^nx=0$ for any $x\in E$, so that $x/1=0$ in $E_P$, leading to $E_P=0$, that is $P\not\in\mathrm{Supp}_R(E)$, a contradiction. To conclude, $\mathrm{Nil}(R)=\sqrt{0:_RE}$.

The converse holds when $E$ is afg. In this case, there is a ring morphism $R\to S$ such that $E$ is an $S$-module with a finite set of generators $\{x_1,\ldots,x_n\}$. Then, $0:_RE=\cap [0:_R x_i\mid i\in\{1,\ldots,n\}]$ which implies $\sqrt{0:_RE}=\cap[\sqrt{0:_R x_i}\mid i\in\{1,\ldots,n\}]$. Assume that $\mathrm{Nil}(R)=\sqrt{0:_RE}$. It follows that any minimal prime ideal $P$ of $R$ contains $\sqrt{0:_R x_i}$ and also $0:_R x_i$ for some $i\in\{1,\ldots,n\}$. This means that $P\in\mathrm{Ass}_R(E)\subseteq\mathrm{Supp}_R(E)$ by Lemma \ref{ass}. Now, $\mathrm{Supp}_R(E)$ is the specialization of $\mathrm{Ass}_R(E)$ which impplies $\mathrm{Supp}_R(E)=\mathrm{Spec}(R)$.

As consequence, we get that, when $E$ is afg, then $\mathrm{Nil}(R)=\sqrt{0:_RE}$ if and only if $\mathrm{Min}(R)\subseteq \mathrm{Supp}_R(E)$. 

(3) The last section is devoted to the study of modules whose supports have only one element. 
\end{remark}

\begin{proposition}\label{FLATOPO} Let $E$ be an $R$-module, such that $\mathrm{Ass}(E)$ is compact in the flat topology. Suppose that there is a multiplicatively closed subset $S$ of $R$, such that $E_S= 0$. Then there is some $a\in S$ such that $E_a= 0$, {\it i.e.} $a \in \mathcal O (E)$. 
\end{proposition}

\begin{proof}To prove this, recall that for an $R$-module $M$, we have $M=0$ if and only if $\mathrm{Ass}(M)= \emptyset$. Moreover, $\mathrm{Ass}_{R_S}(E_S)$ is the set of all $P_S$ such that $P\in \mathrm{Ass}(E)$ and $P\cap S = \emptyset$, so that $E_S= 0$ if and only if $\mathrm{Ass}(E) \subseteq \cup [\mathrm V(s) \mid s \in S]$. Since each $\mathrm V(s)$ is open in the flat topology, there are finitely many elements $s_1,\cdots ,s_n \in S$, such that $\mathrm{Ass}(E) \subseteq \mathrm V (s_1\cdots s_n)$. Setting $a= s_1\cdots s_n$, we get that $E_a=0$ and $a\in S$. 
\end{proof}

\begin{proposition} \label{PROCO}$\mathrm{Supp}_R(E)$ is Zariski closed if and only if $\mathrm{Supp}_R(E)$ is pro-constructible. This holds in the following cases:
\begin{enumerate}
\item  $\mathrm{Supp}_R(E)$ is finite; 
\item $\mathrm{Ass}(E)$ is compact in the flat topology, for example if it is either pro-constructible or finite or closed;

\item $E$ has a finite length; 

\item $E$ is afg over $R$, in which case $\mathrm{V}(\mathcal O(E)) =\mathrm{Supp}_R(E)= \mathrm V(0:_RE)$ and then $\mathcal O(E) = \sqrt{0:_RE}$.
\end{enumerate} 
\end{proposition}
\begin{proof}
 
Since a support is stable under specialization, it is closed when pro-constructible \cite[Corollaire 7.3.2, p.339]{EGA}. 
 
 (1) Obvious.

(2) Suppose that $\mathrm{Ass}(E)$ is compact in the flat topology and $P$ is a prime ideal of $R$. We can apply Proposition \ref{FLATOPO} to the mutiplicatively closed subset $R\setminus P$. It follows that $P$ does not belong to $\mathrm{Supp}(E)$ if and only if  there is some $a \in R\setminus P$, such that $a\in \mathcal O(E)$. Therefore, $\mathrm{Supp}(E) =\mathrm V(\mathcal O(E))$ is closed.  
 
(3) Suppose that $E$ has a finite length. There is a sequence of submodules $0 =E_0 \subset E_1 \subset \cdots \subset E_i \subset \cdots  E_{n-1} \subset E_n = E$, for some positive integer $n$, where each $E_i/E_{i-1}$ is a simple module. Since we have $\mathrm{Supp}(N) = \mathrm{Supp}(M)\cup \mathrm{Supp}(N/M)$, when $M$ is a submodule of $N$, it follows that $\mathrm{Supp} (E) =\cup[ \mathrm{Supp}(E_i/E_{i-1})\mid i= 1,\dots ,n]$ is a finite set of maximal ideals of $R$, because each module $E_i/E_{i-1}$ is  simple.
 
(4) Assume that there is a ring morphism $R\to S$, such that $E$ is an $S$-module with a finite set of generators $\{x_1,\ldots,x_n\}$. Suppose that there is a prime ideal $P$ of $R$, such that $\mathcal O(E) \subseteq P$ and $E_P= 0$. There are elements $r_1,\ldots, r_n \in R\setminus P$, such that $r_ix_i =0$. Setting $r =r_1\cdots r_n\in R \setminus P$, we get that $r \in 0:E \subseteq \mathcal O(E)\subseteq P$, a contradiction. We therefore have  $\mathrm V(\mathcal O(E)) \subseteq \mathrm{Supp}_R(E)$. The reverse inclusion follows from  Theorem~\ref{cars}. 
 Considering the  above equation $\overline{\mathrm{Supp}_R(E)} = \mathrm V (\cap [\sqrt{0:_R x} \mid x\in E])$, we can replace $\cap [\sqrt{0:_Rx} \mid x\in E]$ with $\cap [\sqrt{0:_R x_i}\mid i\in \{1,\ldots, n \}]$, which in turn is equal to $\sqrt{0:_RE}$, because $\sqrt{I\cap  J} = \sqrt I  \cap \sqrt J$ holds for two ideals $I$ and $J$ of $R$.
\end{proof}

\begin{remark} (1) Let $R$ be a ring such that $\mathrm{Ass}(R)$ is finite and $E$ a flat $R$-module. Then $\mathrm{Ass}(E)$ is finite. This follows from \cite[Proposition 2.2, p.94]{L}, because an element $P$ of $\mathrm{Ass}(E)$ is a union of elements of $\mathrm{Ass}(R)$. Therefore,  $\mathrm{Supp}_R(E)$ is Zariski closed

(2) Suppose that $\mathrm{Spec}(R)$ is Noetherian for the flat topology, then so is $\mathrm{Ass}(E)$ for any $R$-module $E$, from which we infer that $\mathrm{Ass}(E)$ is compact for the flat topology and therefore $\mathrm{Supp}(E)$ is closed 
 by Proposition \ref{PROCO}.

 We characterized the rings $R$ whose flat topology is Noetherian. They are the $g$-rings (such that for each $P\in \mathrm{Spec}(R)$, we have $\{Q \in \mathrm{Spec}(R) \mid Q \subseteq P\} = \mathrm D(f)$ for some $f \in R$ \cite[V, Proposition 4]{PG}). These rings are semi-local.
\end{remark}

 For an extension $R\subseteq S$ where $S$ is an algebra of finite type, $\mathrm{Supp}(S/R)$ is closed:
\begin{proposition}\label{algfin} Let $ R\subseteq S$ be a   ring extension, which is
an algebra of finite type, generated by $\{x_1,\ldots , x_n\}$ and $I =Sx_1+\cdots +Sx_n$.   Then:  
\centerline{$\mathrm{Supp}_R(S/R) = \mathrm{V}(\mathcal O(R,S))= \mathrm V((R:S))=\mathrm{V}(R:_R I)$}
It follows that for $P\in \mathrm{Spec}(R)$, $R_P = S_P$ if and only if there is some $a\in R\setminus P$, such that $R_a=S_a$.
 \end{proposition}

\begin{proof}  
Suppose that there exists $P\in \mathrm{Spec}(R)$, such that $\mathcal O(R,S) \subseteq P$ and $R_P= S_P$. There are elements $r_1,\ldots, r_n \in R\setminus P$, such that $r_ix_i \in R$. Setting $r =r_1\cdots r_n\in R \setminus P$, we get that $R_r = S_r$, so that $r \in \mathcal O(R,S)\subseteq P$, a contradiction. We therefore have $\mathrm V(\mathcal O(R,S)) \subseteq \mathrm{Supp}_R(S/R)$.
The reverse  inclusion follows from Theorem~\ref{cars}. 
\end{proof} 
 \begin{remark}A ring extension $R\subseteq S$ is called an FCP extension if the poset $([R,S], \subseteq)$ of all $R$-subalgebras of $S$ is both Artinian and Noetherian, or equivalently, each  chain of $[R,S]$ is finite  \cite{QP}. Let $ R\subseteq S$ be an FCP    extension. Then, $\mathrm{Supp}_R(S/R)=\overline{\mathrm{Ass}_R (S/R)}$ is a finite closed subset of $ \mathrm{Spec}(R)$ \cite[Proposition 4.1 (a)]{DPP3}.
\end{remark}  

Recall that if $E$ is a $T$-module and if $f:R \to T$ is a ring morphism, then $E$ is an $R$-module denoted by $E_{[R]}$. Now note that for $x\in E$, we have $0:_Rx =f^{-1}(0:_Tx)$. The next result is clear by \cite[Proposition 1.2.2, p. 196]{EGA}.

\begin{proposition} Let $E$ be a $T$-module and  $f:R \to T$  a ring morphism, then $\mathcal O_R(E) =f^{-1}(\mathcal O_T(E))$; so that,
 $\overline{ \mathrm{Supp}_R(E)} = \overline{ {}^ af(\overline{\mathrm{Supp}_T(E)})}$.
  It follows that if ${}^af$ is closed, for example if $f$ is integral, we have
  $\overline{ \mathrm{Supp}_R(E)} =  {}^ af(\overline{\mathrm{Supp}_T(E)})$.
\end{proposition}
 
We now examine the behavior of the associated prime ideals, with respect to the previous context.
\begin{lemma} Let $E$ be a $T$-module and  $f:R \to T$  a  flat ring epimorphism. Then we have $\mathrm{Ass}_R (E) ={}^af(\mathrm{Ass}_T(E))$.
\end{lemma}

\begin{proof} We know that for $x\in E$, we have $0:_Rx =f^{-1}(0:_Tx)$, from which we deduce that $0:_Tx = (0:_Rx)T$ by \cite[Proposition 2.1, p. 111]{L}. It follows that $R/(0:_Rx) \to T/(0:_Tx) =T\otimes_R R/(0:_Rx)$ is an injective flat epimorphism. Now it is known that if $g:A \to B$ is an injective ring morphism then $\mathrm{Min}(A) \subseteq {}^ag(\mathrm{Min}(B))$. Moreover, a flat ring morphism has Going-Down, whence minimal prime ideals are lying over minimal prime ideals.
\end{proof}

We observe the following result \cite[Lemme 3.4.4, Section I]{R}:

\begin{lemma} Let $E$ be a $T$-module and $f:R \to T$ a finite ring morphism. Then we have $\mathrm{Ass}_R (E) ={}^af(\mathrm{Ass}_T(E))$.
\end{lemma}

 The two preceding results combine to yield.

\begin{proposition} Let $E$ be a $T$-module and $f:R \to T$ a quasi-finite ring morphism or an FCP   extension. Then we have $\mathrm{Ass}_R (E) ={}^af(\mathrm{Ass}_T(E))$.
\end{proposition}

\begin{proof} In each case $R\to T$ is a tower $R\to S \to T$, where the first morphism is finite and the second a flat epimorphism. When $R \to T$ is quasi-finite, it is enough to apply the main theorem of Zariski \cite[Chapitre IV]{R}. For an FCP extension look at \cite{QP}. 
\end{proof}

\section{Base changes and Oda ideals}

Let $g: R\to T $ be a ring morphism. We say that $g$ verifies the condition (O) if an $R$-module $E$ is zero as soon as $E\otimes_RT = 0$. An extension $M\subseteq N$ of $R$-modules is called pure if $M\subseteq N$ remains injective under a tensorization by an arbitrary module. An extension $R\subseteq T$ is called pure if $R$ is a pure $R$-submodule of $T$ \cite{OL3}. For example, a faithfully flat morphism is a pure extension. Pure ring extensions verify the condition (O).
A finite injective ring morphism has property (O) \cite[Section 2]{RG}.

In case there is a ring morphism $R\to S$ and $E$ and $F$ are $R$-modules, there is an isomorphism $(E\otimes_RF)\otimes_RS \simeq (E\otimes_RS)\otimes_S(F\otimes_RS)$. As we are only interested in the zero property, we will not precise the base ring. The same remark holds for the associativity of the tensor product. Moreover, if $f: R\to T$ is a ring morphism and $\Sigma$ a multiplicatively closed subset of $R$, we can identify $T_{f(\Sigma)}$ with $T_\Sigma = T\otimes_RR_\Sigma$.

\begin{proposition} Let $M$ be an $R$-module,  $f : R \to T$ be a ring   base change and let the $T$-module $N: =M\otimes_RT$.

\begin{enumerate}
\item  $\mathcal O_R(M) \subseteq f^{-1}(\mathcal O_T(N))$. 

\item If in addition  $f$ is injective,  $\mathcal O_R(M) = f^{-1}(\mathcal O_T(N))$  holds if $M$ is either flat  or if $R\subseteq T$ is either pure or finite.
\end{enumerate}
\end{proposition}

\begin{proof}
We first observe that for $a\in R$, we have $N_{f(a)} \cong M_a \otimes_{R_a} T_a$ and (1) is  clear. Now, purity and finiteness are universal properties, that is stable under any base change.
\end{proof}

D. Ferrand (in his thesis \cite{FER}) and J.P. Olivier (in \cite{OL2}) defined and studied absolutely flat ring morphisms as flat ring morphisms $R \to T$, whose co-diagonal   morphisms $T\otimes_RT \to T$ are flat. The reader may find a summary of the properties of absolutely flat morphism at the beginning of \cite{ARAB}. Flat epimorphisms are absolutely flat, because the co-diagonal morphism of an epimorphism is an isomorphism \cite[Lemme 1.10,  p.108]{L}. Etale morphisms and (strict) Henselizations of a local ring are absolutely flat. Flat separable ring extensions are absolutely flat since their co-diagonal morphisms define  projective modules. Note that absolute flatness is a universal property.

In the proof of the two following results, we use that if $f: R\to T$ is a ring morphism with spectral map ${}^af: \mathrm{Spec}(T) \to \mathrm{Spec}(R)$, $I$ is an ideal of $R$ and $J$ an ideal of $T$, then ${}^af^{-1}(\mathrm V(I)) = \mathrm V (IT)$ and $\overline{{}^af(\mathrm V(J))} = \mathrm V(f^{-1}(J))$ \cite[Proposition 1.2.2, p.196]{EGA}.
 
\begin{proposition} \label{flat}If $E$ is an $R$-module and $g: R\to T$ is a flat ring morphism, then ${}^ag^{-1}(\mathrm{Supp}_R(E)) = \mathrm{Supp}_T(E\otimes_RT)$.  In particular, if $\mathrm{Supp}_R(E)$ is closed, so is $\mathrm{Supp}_T(E\otimes_RT)$. In this case, $\mathcal O_T(E\otimes_RT)= \sqrt{\mathcal O_R(E)T}$. If in addition, $g$ is absolutely flat, then $\mathcal O_T(E\otimes_RT)= \mathcal O_R(E)T$. 
 For example, it holds when $f$ is a flat epimorphism.
\end{proposition} 
\begin{proof}

We are concerned with the $T$-module $E\otimes_RT$. The proof is a consequence of the following facts. Let $Q$ be a prime ideal of $T$, lying over $P$ in $R$. Then $R_P\to T_Q$ is faithfully flat, whence has Property (O), and $(E\otimes_RT)_Q\simeq  E\otimes_R T_Q \simeq E\otimes_R (R_P \otimes_{R_P}T_Q) \simeq E_P\otimes_{R_P}T_Q$ by the associativity of tensor products. For the last statement, when $g$ is absolutely flat, it is enough to use the following fact: if $R$ is a reduced ring, so is $T$ \cite[Corollary 2, p.51]{OL2}. It follows that if $I$ is a semi-prime ideal of $R$, then so is $IT$ in $T$, because $T\otimes_R (R/I)\simeq T/IT$ and $R/I \to T/IT$ is absolutely flat.
\end{proof}

\begin{corollary} Let $E$ be an $R$-module and $g: R\to T$ a flat ring morphism of finite presentation, as an $R$-algebra (for example, if $g$ is an injective flat epimorphism of finite type \cite[Theorem 1.1]{CR}), then ${}^ag^{-1}(\overline{\mathrm{Supp}_R(E)}) = \overline{\mathrm{Supp}_T(E\otimes_RT)}$. It follows that $\mathcal O_T(E\otimes_RT) =
\sqrt{\mathcal O_R(E)T}$. If in addition, $g$ is absolutely flat, then $\mathcal O_T(E\otimes_RT)= \mathcal O_R(E)T$.
\end{corollary}
\begin{proof} The spectral map of $g$ is open by the Chevalley theorem \cite[Proposition 7.3.10, p.341]{EGA}. It is enough to observe that, in this case, for any subset $X$ of $\mathrm{Spec}(R)$, we have ${}^ag^{-1}(\overline X) = \overline{{}^ag^{-1}(X)}$ \cite[2.10.1, p.70]{EGA}.
 \end{proof}

\begin{remark} In the two above results applied to $g:R \to R_\Sigma$, where $\Sigma$ is a multiplicatively closed subset of $R$, we have $\mathcal O_{R_\Sigma}(E_\Sigma) = \mathcal O_R (E)_\Sigma$.
\end{remark}

\begin{remark}\label{SUPG} The following remark may be useful. Let $f: R\to S$ be a ring morphism. We may define a ``support" $\mathrm{Supp}(f)$ as the set of all prime ideals $P$ of $R$, such that $R_P \to S_P$ is not an isomorphism. Note that if $f$ is an extension, we recover $\mathrm{Suppp}_R(S/R)$. Suppose now that $f$ is a flat epimorphism. In this case $\mathrm{Supp}(f)= \{P \in \mathrm{Spec}(R) \mid PS= S\}$ by \cite[Proposition 2.4, p.111-112]{L}. If we deal with with an extension, we get that $\mathrm{Supp}_R(S/R) = \{P \in \mathrm{Spec}(R) \mid PS= S\}$.
\end{remark}
\begin{lemma}\label{GEN} Let $f: R \to T$ be a universally generizing (going-down) ring morphism (for  example flat). Then $\mathrm{Spec}(T,R) =\{P\in \mathrm{Spec}(R) \mid PT\neq T\}$.
\end{lemma}
\begin{proof} If $P$ is a prime ideal of $R$ such that there is some $Q\in \mathrm{Spec}(T)$ lying over $P$, then clearly $PT\neq T$. Conversely, if $PT \neq T$, there is a minimal prime ideal in $T/PT$, which lies over a minimal prime ideal of $R/P$, because $R/P \to T/PT$ has going-down. This achieves the proof.
\end{proof}

\begin{proposition} If $f: R \subseteq T$ is an extension such that $f$ is a flat epimorphism and $\mathrm{Supp}(f)$ is closed (for example if $f$ is of finite type), then $\mathrm{Spec}(T,R) = \mathrm D (\mathcal O(R,T))$ is an open affine subset; so that, $\mathcal O(R,T)$ is the radical of a finitely generated ideal of $R$. Moreover, we have $\mathcal O(R,T)T = T$ and $T$ is isomorphic to the ring of sections $\Gamma (\mathrm D(\mathcal O(R,T)),R)$.
\end{proposition}
\begin{proof} By Lemma \ref{GEN}, we have ${}^af(\mathrm{Spec}(T)) =\{P\in \mathrm{Spec}(R) \mid PT\neq T\}$. Moreover, $\{P\in \mathrm{Spec}(R) \mid PT\neq T\} =  \mathrm{Spec}(R) \setminus \mathrm{Supp}(T/R) = \mathrm D(\mathcal O(R,T))$, according to Remark \ref{SUPG}. 

Since ${}^af$ is an injection \cite[Proposition 1.5, p.109]{L}, $ \mathrm D (\mathcal O(R,T)T) =\mathrm{Spec}(T) $, which completes the proof of the first statement. By \cite[Proposition 2.5, p.112]{L}, $\mathrm D(\mathcal O(R,T))$ is an affine open subset of $\mathrm{Spec}(R)$, and $T\simeq \Gamma(\mathrm D(\mathcal O(R,T),R))$.
\end{proof}

We recover  part of  the well known result: an open immersion of affine schemes  arises from a flat epimorphism of finite presentation.

\begin{example}\label{example} A ring extension $R\subset S$ is called minimal if the only $R$-subalgebras of $S$ are $R$ and $S$. In that case, $\mathrm{Supp}(S/R) = \{M\}$, where $M$ is a maximal ideal of $R$ \cite[Th\'eor\`eme 2.2]{FO}, called the crucial ideal of the extension. 

 We recall the following result of \cite[Proposition 4.1]{DPP3} for a ring extension $R \subseteq S$ such that there is a finite maximal chain $R=R_0 \subset\cdots\subset R_i \subset\cdots \subset R_n=S$ in $[R,S]$, where each $R_i\subset R_{i+1}$ is a minimal extension with crucial maximal ideal $M_i$ of $R_i$ for $0\leq i \leq n-1$. $($For instance, such a chain exists if $R \subseteq S$ has $\mathrm{FCP}$.$)$ Then $S$ is finitely generated over $R$ and $\mathrm {Supp}(S/R)=  \overline{\mathrm{Ass}_f(S/R)}$ is a finite closed subset of $\mathrm{Spec}(R)$.
In fact, $\mathrm {Supp}(S/R)=\{M_i\cap R\mid i=0,\ldots,n-1\} = \mathrm{V}(R:_R (s_1,\ldots,s_g))$, where $\{s_1,\ldots,s_g\}$ is any finite set of generators of the $R$-algebra $S$.

If in addition $R\subseteq S$ is integral (whence  finite), then $\mathrm {Supp}(S/R) \subseteq \mathrm{Max}(R)$, so that $\mathrm {Supp}(S/R) =\mathrm V(M_1\cap \cdots \cap M_n \cap R)$ and $\sqrt{(R:S)}= \mathcal O(R,S)= M_1\cap R \cap \cdots M_n\cap R$ is an intersection of finitely many maximal ideals. 
\end{example}

\section{The Constructible support of a module}

If $P$ is a prime ideal of a ring $R$, we denote by $\kappa (P)$ the residue field $R_P/PR_P$ of $R$ at $P$. Now if $f:R\to S$ is a ring morphism, $Q$ is a prime ideal of $S$ and $P:= f^{-1}(Q)$, there is a residue field extension $\kappa (P) \to  \kappa (Q)$.

In order to introduce the notion of constructible support, we will make some observations. Let $f:R \to S$ be a ring morphism and $P$ a prime ideal of $R$. The fiber of $f$ at $P$ is ${}^af^{-1}(\{P\})$ also denoted by ${}^af^{-1}(P)$. The natural map $S \to S\otimes_R\kappa (P)$ defines an homeomorphism $\mathrm{Spec}(S\otimes_R\kappa (P)) \to {}^af^{-1}(P) \subseteq \mathrm{Spec}(S)$. Since the spectrum of a ring is empty if and only if the ring is zero, we get that ${}^af^{-1}(P) \neq \emptyset \Leftrightarrow S\otimes_R\kappa (P) \neq 0$. It follows that $\mathrm{Spec}(S,R) = \{P\in \mathrm{Spec}(R) \mid S\otimes_R \kappa (P) \neq 0\}$. 

Now if $E$ is an $R$-module, Olivier defined the constructible support of $E$ as $\mathrm{CSupp}_R(E):= \{P \in \mathrm{Spec}(R) \mid E\otimes_R\kappa(P) \neq 0\}$, \cite[Chapitre II]{OL3}. We remark that for a ring morphism $f: R \to S$, defining $S$ as an $R$-module, $\mathrm{CSupp}_R(S) = \mathrm{Spec}(S,R)$ is pro-constructible. This may not hold for an arbitrary module $E$. We begin by recalling the following results because \cite{OL3} is perhaps not easily available, some of them coming from the fact that a tensor product of two vector spaces over a field is zero if and only if one of them is zero.

\begin{proposition}\label{PROPER} \cite[1.2, ii, iii) vi), vii),1.3, p.24 and Corollaires 2.7 and 2.8, p.26]{OL3} The following statements hold for an $R$-module $E$:

\begin{enumerate}  
\item  If $E$ is finitely generated over $R$, then $\mathrm{CSupp}(E)=\mathrm{Supp}(E)$;

\item If $E$ is a flat $R$-module,  $\mathrm{CSupp}(E) = \{P \in \mathrm{Spec}(R) | PE \neq E\}$;

\item If $G$ is a pure $R$-submodule of $E$, then $\mathrm{CSupp}(E) = \mathrm{CSupp}(G) \cap \mathrm{CSupp}(E/G)$;

\item $\mathrm{CSupp}(E\otimes_R F) = \mathrm{CSupp}(E) \cap \mathrm{CSupp}(F)$, when $E$ and $F$ are two $R$-modules;

\item If $f:R \to S$ is a ring morphism and $E$ is an $R$-module, then 
${}^a f^{-1}(\mathrm{CSupp}_R(E)) = \mathrm{CSupp}_S(E\otimes_RS)$ and ${}^af(\mathrm{CSupp}_S(E\otimes_R S)) = \mathrm{CSupp}_R(E) \cap \mathrm{CSupp}_R(S)= \mathrm{CSupp}_R(E) \cap \mathrm{Spec}(S,R)$;

\item  $[\mathrm{CSupp}(E)]^{\uparrow}=\mathrm{Supp}(E)=\mathrm{V}(0:E)$;

\item If $E$ is an afg $R$-module, then $\mathrm{CSupp}_R(E)$ is pro-constructible.
\end{enumerate}
\end{proposition}

\begin{remark}\label{PROPER1} When $E$ is finitely generated over $R$, we deduce from (1) and (5) the known result: ${}^a f^{-1}(\mathrm{Supp}_R(E)) = \mathrm{Supp}_S(E\otimes_RS)$ (\cite[Proposition 19, p. 135]{B1}. 
\end{remark}

\begin{proposition}\label{afg} Let $E$ be an afg $R$-module and let $R\to T$ be a ring morphism. Then, $E':=E\otimes_R T$ is an afg $T$-module.
\end{proposition}
\begin{proof} Since $E$ is an afg $R$-module, there exists a ring morphism $R\to S$ such that $E$ is finitely generated over $S$. Let $\{f_i\}_{i=1}^n$ be a generating set of $E$ over $S$. Set $S':=S\otimes_R T$. We have the following commutative diagram

\centerline{$\begin{matrix} 
       R           & \to &        S         \\
\downarrow &  {}  & \downarrow \\
       T          &  \to &        S'  
    \end{matrix}$}
\noindent Then, $E'$ is finitely generated over $S'$ with $\{f_i\otimes 1\}_{i=1}^n$ as a generating set over $S'$, so that $E'$ is an afg $T$-module.
\end{proof}

\begin{proposition}\label{afg1} Let $E$ be an afg $R$-module and let $f:R\to S$ be a generizing (going-down) ring morphism. Then, ${}^a f^{-1}(\mathrm{Supp}_R(E)) = \mathrm{Supp}_S(E\otimes_RS)$.
\end{proposition}
\begin{proof} Since $E$ is an afg $R$-module, $X:=\mathrm{CSupp}_R(E)$ is pro-constructible. Moreover, $E':=E\otimes_R S$ is an afg $S$-module. Now, using \cite[Corollaire 7.3.2 and Proposition 7.3.3, p. 339]{EGA}, we get $\mathrm{Supp}(E')=[\mathrm{CSupp}(E')]^{\uparrow}=\overline{\mathrm{CSupp}(E')}=\overline{{}^a f^{-1}(\mathrm{CSupp}(E))}=\overline{{}^af^{-1}(X)}={}^af^{-1}(\overline X)={}^af^{-1}( X^{\uparrow})={}^af^{-1}(\mathrm{CSupp}_R(E)^{\uparrow})={}^af^{-1}(\mathrm{Supp}_R(E))$ because $\overline X=X^{\uparrow}$ since $X$ is pro-constructible. 
\end{proof}

Each ring $R$ admits a universal absolutely flat ring $T(R)$, such that there is a ring epimorphism $t: R \to T(R),$ verifying: for each ring morphism $R \to S$, where $S$ is absolutely flat, there is a (unique) ring morphism $T(R) \to S$, such that $R \to S = R\to T(R) \to S$ \cite{OL1}. The Zariski topology and the constructible topology on $\mathrm{Spec}(T(R))$ coincide. The spectral map ${}^at :\mathrm{Spec}(T(R)) \to \mathrm{Spec}(R)$ is an homeomorphism when $\mathrm{Spec}(R)$ is endowed with the constructible topology. Moreover, for each prime ideal $Q$ of $T(R)$ lying over $P$ in $R$, the residual extension $\kappa (P) \to \kappa (Q)$ is an isomorphism of fields and $ \kappa (Q) \cong  \mathrm T (R)_Q$. 

Let $E$ be an $R$-module and the flat $T(R)$-module $F:= E\otimes_R T(R)$. Because $\mathrm T(R)$ is absolutely flat, $\mathrm{CSupp}_{T(R)}(F)= \mathrm{Supp}_{T(R)}(F)$. Since $F$ is flat over $T(R)$, $ \mathrm{CSupp}_{T(R)}(F) = \{ M\in \mathrm{Spec}(T(R)) \mid MF \neq F\}$ holds according to Proposition~\ref{PROPER}(2). Now by Proposition~\ref{PROPER}(5) we have $ {}^at^{-1}(\mathrm{CSupp}_R(E)) = \mathrm{Supp}_{T(R)}(F)$ and ${}^at(\mathrm{Supp}_{T(R)}(F)) =$

\noindent$ \mathrm{CSupp}_R(E) $, because ${}^at$ is surjective. Hence, ${}^at$ induces an homeomorphism  $\mathrm{Supp}_{T(R)}(F) \to\mathrm{CSupp}_R(E)$ for the constructible topology. Example~\ref{EXA} tells us that  $\mathcal O_{T(R)}(F) = 0:_{T(R)}F$ if $\mathrm{Supp}_{T(R)}(F)$ is closed.
 It follows that $\overline{\mathrm{CSupp}_R (E)} =\overline{{}^at(\mathrm V(0:F))}  = \mathrm V (t^{-1}(0:F))$. 

\section{ Oda ideals and algebraic  constructions}

We now examine the properties  of Oda ideals with respect to some constructions.

\begin{proposition}\label{pro} Let $E$ be an $R$-module of finite type, then 
$\mathcal O(E) =\mathcal O_R(R/(0:E))$. 
\end{proposition}

\begin{proof} 
Observe that $a\in \mathcal O(E)$ if and only if $R_a = (0:E)R_a$ by the Stokes formula.
\end{proof}

The statements of the next Proposition are easy to prove.

\begin{proposition}\label{con} The following statements hold:
\begin{enumerate}
\item  Let $N$ be a submodule of an $R$-module $M$, then;

\centerline{$\mathcal O(M) = \mathcal O (N) \cap \mathcal O(M/N)$.}

\item If $R\subseteq S \subseteq T$ is a tower of extensions, then:

\centerline{$\mathcal O_R(R,T) = \mathcal O_R(R,S) \cap \mathcal O_S(S,T)$.} 
\noindent  This equation is related to the  well known equation:
 
 \centerline{$\sqrt{(R:_RT)} = \sqrt{(R:_RS)}\cap \sqrt{(S:_ST)}$.}

\item Let $E$ be an $R$-module which is an upward directed union of finitely generated submodules, $\{E_i \mid i\in  I\}$, then :

\centerline{$\mathcal O(E) = \cap [\mathcal O(E_i) \mid i \in I] = \cap [\sqrt{0:_R E_i} \mid i\in I]$;}

\centerline{$\mathrm{Supp}(E) = \cup [\mathrm{Supp}(E_i) \mid i\in I]= \cup [\mathrm V(0:E_i) \mid i\in I]$.} 

\item Similarly, let $R \subseteq S$  be a ring extension, then $S$ is the upward directed  union of all its $R$-subalgebras of finite type $S_i$ for $i\in I$ and then $\mathcal O(R,S) =  \cap [\mathcal O(R,S_i) \mid i \in I]$.
\end{enumerate}
\end{proposition} 

\section{Critical ideals and crucial ideals}

We say that an $R$-module $E$ has a crucial ideal $\mathcal C (E)$ if $|\mathrm{Supp}(E)| =1$ and then  the only element $M$ of $\mathrm{Supp}(E)$ is a maximal ideal because a support is stable under specialization, so that $\mathcal C(E):= M$. It is then clear that $\mathcal O(E) = M$ is a maximal ideal and that $\mathrm{Supp}(E)$ is closed.

We also say that a ring extension $R\subset S$ has a crucial ideal if $\mathrm{Supp}(S/R)$ has a unique element, which is necessarily a maximal ideal $M$, called the crucial ideal of $R\subseteq S$ \cite[Definition 2.1]{Pic 15} and \cite[Definition 2.10]{POINT}. This means that $R_P=S_P$ for each prime ideal $P\neq M$. We also say that the extension is $M$-crucial and we set $\mathcal C(R,S) := M$.

 We observed that  a minimal extension has a crucial ideal (Example \ref{example}). More generally, a pointwise minimal extension has also a crucial ideal.
An extension $R\subseteq S$ is called pointwise minimal if $R\subset R[x]$ is a minimal extension for each $x\in S\setminus R$ \cite[Theorem 3.2]{POINT}.

\begin{proposition}\label{1.191}  Let $R \subset S$ be an extension, with conductor $C:= (R:S)$. The following statements hold:
 
 (1) If $R\subset S$ is $M$-crucial, then $C \subseteq M$.
 
 (2)  If  $R\subset S$ is integral, then $R\subset S$  has a crucial  ideal if and only if $\sqrt C \in \mathrm{Max}(R)$,  and then  $\mathcal C (R,S)= \sqrt C$.
\end{proposition}

\begin{proof}   

(1) If the extension is $M$-crucial, suppose that there is some $x \in C\setminus M$, then it is easily seen that $R_M=S_M$, a contradiction.

 (2) We denote by   $ \{R_\alpha \mid \alpha \in I\}$   the family of all  finite subextensions 
$R \subset R_\alpha $  of $R\subset S$  and  set $C_\alpha:= (R:R_\alpha)$.
For  $M \in \mathrm{Spec}(R)$,  observe that $M$ is a crucial ideal of  $R \subset S$    if and only if  $M$ is a crucial ideal of  each $R\subset R_\alpha$. Then it is enough to use the following facts: $\mathrm{Supp}(R_\alpha /R) = \mathrm{V}(C_\alpha)$ and $C = \cap [C_\alpha \mid \alpha \in I]$.  
\end{proof}

We now introduce the notion of a critical ideal of an $R$-module $E$ that generalizes the critical ideal of an extension defined in \cite[p. 1093]{LC}.

\begin{definition} Let $E$ be an $R$-module. We say that an ideal $J$ is critical for $E$ if $J= \sqrt{0:_R x}$ for each $x \in E\setminus \{0\}$. If such an ideal exists, it is unique and is a prime ideal (it is enough to mimic the proof of \cite[Lemma 2.11]{LC}).
\end{definition}

A critical ideal of a ring extension $R \subseteq S$ is a critical ideal of the $R$-module $S/R$. 

The following result is clear.

\begin{lemma}\label{6.3} If an $R$-module $E$ has a critical ideal $J$, then $\mathcal O(E) = J$ is a prime ideal, 
 $\mathrm{Ass}(E)=\{J\}$ and $\mathrm{Supp}(E)=\mathrm{V}(J)$
  is closed. Moreover, if $J\in \mathrm{Max}(R)$, then $J$ is also the crucial ideal of $E$. 
\end{lemma}

\begin{lemma}\label{6.4} Let $E$ be an $R$-module. Then $\mathrm{Ass}(E)$ is a chain if and only if $\sqrt{0:x}$ is a prime ideal for each $x\in E\setminus\{0\}$. If this statement holds, then $\mathrm{Ass}(E)=\{\sqrt{0:x}\mid x\in E\setminus\{0\}\}$. Moreover,  $\mathcal O(E)$ is a prime ideal and the least element of $\mathrm{Ass}(E)$.
\end{lemma}
\begin{proof} Assume first that $\sqrt{0:x}$ is a prime ideal for each $x\in E\setminus\{0\}$. Let $x,y\in E\setminus\{0\}$. Set $z:=x+y,\ M:=\sqrt{0:x},\ N:=\sqrt{0:y}$ and $P:=\sqrt{0:z}$.    If $z=0$, then $y=-x$ gives $M=N$. So we may assume that $z\neq 0$. We claim that $M,N$ and $P$ are not all distinct. Mimicking the proof of \cite[Proposition 2.14 (3)]{LC}, we get that each ideal contains the intersection of the two others. Then we have $M\cap N\subseteq P\ (1),\ N\cap P\subseteq M\ (2)$ and $P\cap M\subseteq N\ (3)$. In particular, (1) leads to either $M\subseteq P\ (*)$ or $N\subseteq P\ (**)$. In case $(*)$, we get $M\subseteq N$ by (3), so that (2) implies to either $N\subseteq M$ and then $M=N$ or $P\subseteq M$ so that $P=M$. Case $(**)$ gives a similar result. Assume  $M\neq N$ and choose $P=M$. By (3), we get $M=P\subseteq N$. If $P=N$, it follows by (2) that $N\subseteq M$. To conclude, $\{\sqrt{0:x}\mid x\in E\setminus\{0\}\}$ is a chain and $\mathcal O(E)$ is a prime ideal as an intersection of a chain of prime ideals. Now, let $Q\in  \mathrm{Spec}(R)$. Then, $Q\in\mathrm{Ass}(E)$ if and only if $Q$ is minimal in the set of prime ideals containing $0:x$ for some $x\in E\setminus\{0\}$, which is equivalent to $Q$ is minimal in the set of prime ideals containing $\sqrt{0:x}$ for some $x\in E\setminus\{0\}$. Since $\sqrt{0:x}$ is a prime ideal for each $x\in E\setminus\{0\}$, we get that $Q=\sqrt{0:x}$ for some $x\in E\setminus\{0\}$, so that $\mathrm{Ass}(E)=\{\sqrt{0:x}\mid x\in E\setminus\{0\}\}$ is a chain whose least element is $\mathcal O(E)$.

Conversely, assume that $\mathrm{Ass}(E)$ is a chain and let $x\in E\setminus\{0\}$. Then, $\sqrt{0:x}$ is the intersection of the elements of $X:=\{P\in\mathrm{Spec}(R)\mid \sqrt{0:x}\subseteq P\}$, and is also the intersection of the minimal elements of $X$.  But these minimal elements are in $\mathrm{Ass}(E)$, which is a chain. Then $\sqrt{0:x}\in \mathrm{Ass}(E)$ and is a prime ideal.
\end{proof}

\begin{corollary}\label{6.4C} Let $R \subset S$ be a ring extension. Then $\mathrm{Ass}(S/R)$ is a chain if and only if $\sqrt{R:x}$ is a prime ideal for each $x\in S\setminus R$, in which case $\mathrm{Ass}(S/R)=\{\sqrt{R:x}\mid x\in S\setminus R\}$. Moreover, $\mathcal O(R,S)$ is a prime ideal and the least element of $\mathrm{Ass}(S/R)$.
\end{corollary}

\begin{proposition}\label{6.5} Let $E$ be an $R$-module. Each of the following statements implies that $E$ has a critical ideal.
\begin{enumerate}

\item $E$ has a crucial ideal $M$ ($M$ is the critical ideal).
\item   $\mathcal O(E)$ is a maximal ideal (the critical ideal is $\mathcal O(E)$).
\item For each $x\in E$, $x\neq 0$, $\sqrt{0:x}$ is a maximal ideal (the critical ideal is $\mathcal O(E)$, which is maximal).

\item For each $x\in E$, $x\neq 0$, $\sqrt{0:x}$ is a prime ideal and $\mathrm{Supp}(E)\subseteq \mathrm{Max}(R)$ (the critical ideal is $\mathcal O(E)$, which is maximal).
\end{enumerate}
\end{proposition}
\begin{proof} (1) Let $M$ be  a crucial ideal of $E$. We have $\mathrm{Supp}(E) =\mathrm V(M)$, because $M$ is a maximal ideal. It follows that $\cap [\sqrt{0:x} \mid x \in E] = M$ and the result follows easily. 

(2) The proof of the second  statement is similar. 

 (3)  Use Lemma \ref {6.4}.

(4) Use (3) and Lemma \ref {6.4} because $\sqrt{0:x}$ is a maximal ideal for each $x\in E$, $x\neq 0$ since $\mathrm{Supp}_R(E) \supseteq \mathrm{Ass}_R(E)$ by Lemma \ref{ass}.
\end{proof}

\begin{corollary}\label{6.7} Let $E$ be an $R$-module and $M\in \mathrm{Max}(R)$. The following conditions are equivalent: 
\begin {enumerate}
\item $M=\mathcal C(E)$; 

\item $M$ is critical for $E$;

\item $M=\mathcal O(E)$.
\end{enumerate}
\end{corollary}
\begin{proof} Use Lemma \ref{6.3} and Proposition \ref{6.5}.
\end{proof}

\begin{proposition}\label{6.8} An FCP integral extension $R\subset S$ has a critical ideal if and only if $R\subset S$ is an $M$-crucial extension with $M\in\mathrm{Max}(R)$. In this case, $M$ is the critical ideal of the extension and $M=\sqrt{(R:R[x])}$ for each $x\in S\setminus R$.
\end{proposition}
\begin{proof} Assume that $R \subset S$ has a critical ideal $M$. According to Lemma \ref{6.3}, we have $\mathrm{Supp}_R(S/R):=\mathrm{V}(M)$ and $M$ is a prime ideal. But, $\mathrm{Supp}_R(S/R)\subseteq\mathrm{Max}(R)$ by \cite[Lemma 3.3]{DPP2}. Then $M\in\mathrm{Max}(R)$ and Corollary \ref{6.7} gives the equivalence. If these conditions hold, then $M=\sqrt{R:x}$ for each $x\in S\setminus R$. Let $x\in S\setminus R$. Clearly, $(R:R[x])\subseteq R:x$, which leads to $\sqrt{(R:R[x])}\subseteq\sqrt{R:x}$. For any $a\in \sqrt{R:x}$, there exists some $n\in\mathbb N$ such that $a^n\in R:x$, so that $a^nx\in R$. As $x$ is integral over $R$, there exists some $p\in\mathbb N$ such that $a^px^r\in R$ for any $r\in\mathbb N$. Then, $a^p\in(R:R[x])$, giving $a\in\sqrt{(R:R[x])}$ and $ \sqrt{R:x}\subseteq\sqrt{(R:R[x])}$. To conclude, $M=\sqrt{(R:R[x])}$ for each $x\in S\setminus R$. 
\end{proof} 

For a ring extension $R\subseteq S$, we denote by $[R,S]$ the set of subextensions of $R\subseteq S$ and by $\overline R$ the integral closure of $R$ in $S$. The length $\ell[R,S]$ of $[R,S]$ is the supremum of the lengths of chains of $R$-subalgebras of $S$. Note that if $R\subseteq S$ has FCP, then there does exist some maximal chain of $R$-subalgebras of $S$ with length $\ell[R,S]$ \cite[Theorem 4.11]{DPP3}.
Recall that an extension $R\subseteq S$ is Pr\"ufer if $R\subseteq T$ is a flat epimorphism for each $T\in[R,S]$.

\begin{corollary}\label{6.9} Let $R\subset S$ be an FCP extension. Then $R \subset S$ has a critical ideal $M$ if and only if $R\subset S$ is an $M$-crucial extension with $M\in\mathrm{Max}(R)$. If these conditions hold, then either $\overline R=S$ or $\overline R\subset S$ is a Pr\"ufer extension of length $|\mathrm{Supp}(S/\overline R)|\leq |\mathrm{V}_{\overline R}(M\overline R)|$. In this last case, $\overline R\subset S$ is locally minimal.
\end{corollary}
\begin{proof} Assume that $R\subset S$ has a critical ideal $M$. Then $M$ is also the critical ideal of $ R\subseteq\overline R $. By Proposition \ref{6.8}, we get that $M\in  \mathrm{Max}(R)$. Then, $R\subset S$ is an $M$-crucial extension by Corollary \ref{6.7}, which gives also the converse. 

Assume that these conditions hold and that $\overline R\neq S$. Since $R\subset S$ has FCP, $\mathrm{V}_{\overline R}(M\overline R)$ is a finite subset $\{M_1,\ldots,M_n\}$ of $\mathrm{Max}(\overline R)$. Moreover, $\mathrm{Supp}_R(S/R)=\mathrm V_R(M)$ implies that $\mathrm{Supp}_{\overline R}(S/\overline R)\subseteq\mathrm V_{\overline R}(M\overline R)\subseteq\mathrm{Max}(\overline R)$. We may assume that $\mathrm{Supp}_{\overline R}(S/\overline R)=\{M_1,\ldots,M_p\}$ with $p\leq n$, after a reordering. According to \cite[Proposition 6.12]{DPP2}, it follows that $\ell[\overline R,S]=|\mathrm{Supp}_{\overline R}(S/\overline R)|=p\leq |\mathrm{V}_{\overline R}(M\overline R)|$. In fact, $\overline R\subset S$ is locally minimal because $\overline R_{M_i}\subset S_{M_i}$ is minimal Pr\"ufer for any $i\in\{1,\ldots,p\}$ by the same reference.
\end{proof} 

\begin{remark}\label{6.10} A critical ideal needs not to be the crucial ideal of an extension. Let $R$ be an integral domain which is not a field and $S:=R[X]$ the polynomial ring in the indeterminate $X$. Of course, for any $P(X)\in S\setminus R$, we have $(R:P(X))=0=\sqrt{(R:P(X))}\in\mathrm{Spec}(R)$, so that $0$ is the critical and the Oda ideal of the extension while it is not the crucial ideal of the extension since $\mathrm{Supp}(S/R)=\mathrm{Spec}(R)\subset \mathrm{Max}(R)$.
\end{remark}

Given a ring $R$, its Nagata ring $R(X)$ is the localization $R(X):=T^{-1}R[X]$ of the ring of polynomials $R[X]$ with respect to the multiplicatively closed subset $T$ of all polynomials with content $R$.
We compute the Oda ideal and also the crucial ideal of a Nagata extension $R(X)\subset S(X)$ when $R\subset S $ is $M$-crucial. Recall that if $R\subset S$ is integral, then $S(X)\cong R(X)\otimes_R S$ \cite[Lemma 3.1]{DPP3}. The same property holds if $R\subset S$ is a flat epimorphism since the surjective map $R(X)\otimes_R S\to S(X)$ is injective. Indeed, $R(X)\to R(X)\otimes_R S$ is a flat epimorphism and $R(X)\to R(X)\otimes_R S\to S(X)$ is injective \cite[Scholium A(3)]{QP}.

\begin{lemma}\label{6.11} Let $R\subset S$ be an $M$-crucial extension such that $S(X) \cong R(X)\bigotimes_RS$ (for example, if $R\subset S$ is integral or a flat epimorphism). Then $R(X) \subset S(X)$ is $MR(X)$-crucial.  
\end{lemma}

\begin{proof} The extension $g:R\to R(X)$ is faithfully flat and $\mathrm{Supp}(S/R)=\{M\}$. Let $Q\in\mathrm{Supp}_{R(X)}(S(X)/R(X))$. Applying \cite[Proposition 2.4 (b)]{DPP3}, we get that $Q\in ({}^{a}g)^{-1}(\mathrm{Supp}(S/R))$, so that ${}^{a}g(Q)=Q\cap R\in\mathrm{Supp}_R(S/R)=\{M\}$, giving $M=Q\cap R$. It follows that $M\subseteq Q$, which implies $MR(X)\subseteq Q$ and then $Q=MR(X)$ since $MR(X)\in\mathrm{Max}(R(X))$. Therefore,  $\mathrm{Supp}_{R(X)}(S(X)/R(X))=\{MR(X)\}$.
\end{proof}

\begin{corollary}\label{6.12} Let $R\subset S$ be an $M$-crucial extension such that $S(X) \cong R(X)\bigotimes_RS$. Then $MR(X)=\mathcal O(R(X),S(X))$ and is the critical ideal of $R(X) \subset S(X)$.
\end{corollary}
\begin{proof} Use Corollary \ref{6.7}.
\end{proof}

\begin{lemma}\label{6.13} Let $R\subset S$ be a ring extension with conductor $C$. Then $(R(X):S(X))=CR(X)$.
\end{lemma}
\begin{proof} Obviously, $CR(X)\subseteq (R(X):S(X))$. Conversely, let $P(X)/Q(X)$

\noindent $\in(R(X):S(X))$, with $P(X),Q(X)\in R[X]$ where the content of $Q(X)$ is $c(Q)=R$. Set $P(X):=\sum_{i=0}^na_iX^{i},\ a_i\in R$ for each $i\in\{0,\ldots,n\}$. For any $s\in S$, we have $(P(X)/1)(s/1)\in R(X)$, so that there exist $U(X),V(X)\in R[X]$ with $c(U)=R$ such that $sU(X)P(X)=V(X)\in R[X]$. In particular, $c(UP)=c(P)$ by the content formula. Set $U(X)P(X):=\sum_{j=0}^pb_jX^{j},\ b_i\in R$ for each $j\in\{0,\ldots,p\}$. Then, $c(P)=\sum_{i=0}^nRa_i=c(UP)=\sum_{j=0}^pRb_j$. But, for any $j\in\{0,\ldots,p\}$, we have $sb_j\in R$ leading to $sa_i\in\sum_{j=0}^pRsb_j\subseteq R$ for each $i\in\{0,\ldots,n\}$, so that $a_i\in (R:S)=C$ for each $i\in\{0,\ldots,n\}$ and $P(X)\in CR[X]$. At last $P(X)/Q(X)\in CR(X)$, giving the wanted equality. 
\end{proof}

\begin{proposition}\label{6.14}  Let $R\subset S$ be an FCP  extension. Then 

\centerline{$\mathcal O_{R(X)}(R(X),S(X))=\sqrt{\mathcal O_R(R,S)R(X))}$.}
\end{proposition}
\begin{proof} Since $R\subset S$ has FCP, so has $R(X)\subset S(X)$ by \cite[Theorem 3.9]{DPP3}. Then, $\mathrm{Supp}_R(S/R)$ and $\mathrm{Supp}_{R(X)}(S(X)/R(X))$ are finite, so that they are Zariski closed by Proposition \ref{PROCO}. Moreover, $R\to R(X)$ is flat and $S(X) \cong R(X)\bigotimes_RS$ by \cite[Corollary 3.5]{DPP3}. It follows from Proposition \ref{flat} that $\mathcal O_{R(X)}(R(X),S(X))=\sqrt{\mathcal O_R(R,S)R(X)}$.
\end{proof}

\begin{corollary}\label{6.15} Let $R\subset S$ be an integral FCP extension. Then 
\centerline{$\mathcal O_{R(X)}(R(X),S(X))=\mathcal O_R(R,S)R(X)$.}
\end{corollary}
\begin{proof} We use the results of Example \ref{example}. Since $R\subset S$ is an integral FCP extension, we get that $\mathrm{Supp}_R(S/R)=\{M_1,\ldots,M_n\}\subseteq \mathrm{Max}(R)$ and $\mathcal O_R(R,S)=\cap_{i=1}^nM_i=\prod_{i=1}^nM_i$. Since $R(X)\subset S(X)$ is also an integral FCP extension by \cite[Theorem 3.4]{DPP3}, it follows that $\mathrm{Supp}_{R(X)}(S(X)/R(X))=\{M_1R(X),\ldots,M_nR(X)\}\subseteq \mathrm{Max}(R(X))$ and $\mathcal O_{R(X)}(R(X),S(X))=\cap_{i=1}^n(M_iR(X))=\prod_{i=1}^n(M_iR(X))=$

\noindent $(\prod_{i=1}^nM_i)R(X)=\mathcal O_R(R,S)R(X)$ according to \cite[Lemma 3.3]{DPP3}. 
\end{proof}

\end{document}